\documentclass[a4paper,11pt]{article}
\usepackage[T1]{fontenc}
\usepackage[utf8]{inputenc}
\usepackage[english]{babel}
\usepackage[stretch=10]{microtype}
\usepackage{mathtools}
\mathtoolsset{centercolon}
\usepackage{amssymb}
\usepackage{tikz}
\usepackage{enumitem}
\usepackage{graphicx}
\usepackage{xcolor}
\usepackage{mathrsfs}
\usepackage[autostyle]{csquotes}
\usepackage[giveninits=true,doi=false,isbn=false,url=false, hyperref,maxbibnames=99,sortcites]{biblatex}
\usepackage{enumitem}
\usepackage{comment}
\usepackage{hyperref}
\hypersetup{
	colorlinks,
	linkcolor={red!50!black},
	citecolor={blue!80!black},
	urlcolor={blue!80!black}
}
\usepackage{amsthm}
\usepackage{geometry}
\usepackage{cleveref}
\theoremstyle{plain}
\newtheorem{theorem}{Theorem}[section]
\newtheorem{corollary}[theorem]{Corollary}
\newtheorem{lemma}[theorem]{Lemma}
\newtheorem{proposition}[theorem]{Proposition}
\theoremstyle{definition}

\theoremstyle{remark}
\newtheorem*{remark}{Remark}
\renewcommand{\epsilon}{\varepsilon}
\renewcommand{\phi}{\varphi}
\newcommand*{\numberset}{\mathbb}
\newcommand*{\N}{\numberset{N}}
\newcommand*{\R}{\numberset{R}}
\DeclarePairedDelimiter\abs{\lvert}{\rvert}
\DeclarePairedDelimiter\norm{\lVert}{\rVert}
\DeclarePairedDelimiterXPP\pnorm[1]{}{\lVert}{\rVert}{_p}{#1}
\DeclarePairedDelimiterXPP\qnorm[1]{}{\lVert}{\rVert}{_q}{#1}
\DeclarePairedDelimiterX{\scalar}[2]{\langle}{\rangle}{#1 , #2}
\newcommand{\disteq}{\stackrel{d}{=}}
\newcommand{\pradial}{p\text{-radial}}
\newcommand{\Kradial}{K\text{-radial}}
\renewcommand{\Pr}{\mathbf{P}}
\DeclareMathOperator{\Ex}{\mathbf{E}}
\DeclareMathOperator{\Var}{\mathbf{Var}}
\DeclareMathOperator{\Cov}{\mathbf{Cov}}
\newcommand{\vol}{\abs}
\newcommand{\pball}{\mathbb{B}_p^n}
\newcommand{\psphere}{\mathbb{S}_p^{n-1}}
\newcommand*\de{\mathop{}\!\mathrm{d}}
\newcommand*{\indic}{\boldsymbol{1}}
\newcommand{\oneover}[1]{\frac{1}{#1}}
\newcommand{\XX}{\mathbf{X}}
\newcommand{\WW}{\mathbf{W}}
\newcommand{\rate}{\mathcal{I}}
\newcommand{\caseif}{\quad\text{if }\,}
\newcommand{\other}{\quad\text{otherwise}}
\newcommand{\bPEX}{\boldsymbol{\norm{P_EX}}}
\newcommand{\borel}{\mathcal{B}}
\newcommand{\G}{\mathbb{G}}
\newcommand{\X}{\mathbf{X}}
\makeatletter
\newcommand{\raisemath}[1]{\mathpalette{\raisem@th{#1}}}
\newcommand{\raisem@th}[3]{\raisebox{#1}{$#2#3$}}
\makeatother
\newcommand{\dconv}{\xrightarrow[\raisemath{3pt}{n\to\infty}]{d}}
\newcommand{\pconv}{\xrightarrow[\raisemath{3pt}{n\to\infty}]{\Pr}}
\newcommand{\asconv}{\xrightarrow[\raisemath{3pt}{n\to\infty}]{\text{a.s.}}}
\begin{document}
	\author{Joscha Prochno \and Christoph Thäle \and Nicola Turchi}
	\title{Geometry of \(\ell_p^n\,\text{-balls}\): \\ Classical results and recent developments}
	\date{}
	\maketitle
	\tableofcontents
	\pagestyle{plain}
	\let\thefootnote\relax
		
	\let\thefootnote\relax
\section{Introduction}

The geometry of the classical $\ell_p$ sequence spaces and their finite-dimensional versions is nowadays quite well understood. It has turned out that it is often a probabilistic point of view that shed (new) light on various geometric aspects and characteristics of these spaces and, in particular, their unit balls. In this survey we want to take a fresh look at some of the classical results and also on some more recent developments. The probabilistic approach to study the geometry of $\ell_p^n$-balls will be an asymptotic one. In particular, our aim is to demonstrate the usage of various limit theorems from probability theory, such as laws of large numbers, central limit theorems or large deviation principles. While the law of large numbers and the central limit theorem are already part of the -- by now -- classical theory (see, e.g., \cite{SS1991,S1998, S}), the latter approach via large deviation principles was introduced only recently in the theory of asymptotic geometric analysis by Gantert, Kim and Ramanan in \cite{GKR2017}. Most of the results we present below are not new and we shall always give precise references to the original papers. On the other hand, we provide detailed arguments at those places where we present generalizations of existing results that cannot be found somewhere else. For some of the other results the arguments are occasionally sketched as well. 

\medskip

Our text is structured as follows. In Section \ref{sec2:Preliminaries} we collect some preliminary material. In particular, we introduce our notation (Section \ref{subsec21:Notation}), the class of $\ell_p^n$-balls (Section \ref{subsec22:LpBalls}), and also rephrase some background material on Grassmannian manifolds (Section \ref{subsec23:Grassmannians}) and large deviation theory (Section \ref{subsec24:LargeDeviations}). In Section \ref{sec3:PMonConvexBodies} we introduce a number of probability measures that can be considered in connection with a convex body. We do this for the case of $\ell_p^n$-balls (Section \ref{subsec31:PMonLpBalls}), but also more generally for symmetric convex bodies (Section \ref{subsec32:ConeMeasureSymmConvBody}). The usage of the central limit theorem and the law of large numbers in the context of $\ell_p^n$-balls is demonstrated in Section \ref{sec4:CLTandLLN}. We rephrase there some more classical results of Schechtman and Schmuckenschl\"ager (Section \ref{subsec41:ClassiclSchechtSchmuck}) and also consider some more recent developments (Section \ref{subsec42:RecentDevelopmentsMultivariateCLT+OutlookMatrix}) including applications of the multivariate central limit theorem. We also take there an outlook to the matrix-valued set-up. The final Section \ref{sec5:LargeDeviations} is concerned with various aspects of large deviations. We start with the classical concentration inequalities of Schechtman and Zinn (Section \ref{subsec51:ClassicalConcentrationIneq}) and then describe large deviation principles for random projections of $\ell_p^n$-balls (Section \ref{subsec52:RecentDevelopmentsLDPs}). 

\section{Preliminaries}\label{sec2:Preliminaries}

In this section we shall provide the basics from both asymptotic geometric analysis and probability theory that are used throughout this survey article. The reader may also consult \cite{IsotropicConvexBodies,AsymptoticGeometricAnalysisBookPart1,DZ, Kallenberg,dH} for detailed expositions and additional explanations when necessary.

\subsection{Notation}\label{subsec21:Notation}

We shall denote with \(\N=\{1,2,\ldots\} \), \(\R\) and \(\R^+ \) the set of natural, real and real non-negative numbers, respectively. Given \(n\in\N \), let \(\R^n\) be the \(n\text{-dimensional} \) vector space on the real numbers, equipped with the standard inner product denoted by \(\langle\cdot\,,\cdot\rangle\). We write \(\borel(\R^n) \) for the \(\sigma\text{-field}\) of all Borel subsets of \(\R^n \). Analogously, for a subset \(S\subseteq \R^n\), we denote by \(\borel(S)\coloneqq\{A\cap S:A\in\borel(\R^n)\}\) the corresponding trace \(\sigma\text{-field}\) of \(\borel(\R^n)\). Given a set \(A\), we write \(\# A\) for its cardinality. For a set \(A\subseteq\R^n\), we shall write \(\indic_A\colon\R^n\to\{0,1\} \) for the indicator function of \(A\). 
Given  \(A\in\borel(\R^n)\), we write \(\vol{A} \) for its $n$-dimensional Lebesgue measure and frequently refer to this as the volume of \(A\).

Given sets \(I\subseteq\R^+\) and \(A\subseteq\R^n\), we define the set \(I A \)  as follows,
\begin{equation*}
IA\coloneqq\{rx\in\R^n:r\in I, x\in A \}.
\end{equation*}
If \(I=\{r\}\), we also write \(r A\) instead of \(\{r\}A \). Note that \(\R^+\! A\) is usually called the cone spanned by \(A\).

We say that \(K\subseteq\R^n\) is a convex body if it is a convex, compact set with non-empty interior. We indicate with \(\partial K \) its boundary.

Fix now a probability space \((\Omega,\mathcal{F},\Pr) \). We will always assume that our random variables live in this probability space. Given a random variable \(X\colon\Omega\to\R^n\) and a probability measure \(Q \) on \(\R^n\), we write \(X\sim Q \) to indicate that  \(Q \) is the probability distribution of \(X\), namely, for any \(A\in\borel(\R^n)\),
	\begin{equation*}
		\Pr(X\in A)=\int_{\R^n}\indic_A(x)\de Q(x).
	\end{equation*}
We write \(\Ex \) and \(\Var \) to denote the expectation and the variance with respect to the probability \(\Pr \), respectively.

Given a sequence of random variables \((X_n)_{n\in\N}\) and a random variable \(Y\) we write 
	\begin{equation*}
		X_n\dconv Y, \qquad X_n\pconv Y, \qquad X_n\asconv Y,
	\end{equation*}
to indicate that \((X_n)_{n\in\N}\) converges to $Y$ in distribution, probability or almost surely, respectively, as \(n\to\infty\).

We write \(N\sim \mathcal{N}(0,\Sigma)\) and say that \(N\) is a centred Gaussian random vector in \(\R^n\) with covariance matrix \( \Sigma\), i.e., its density function w.r.t. the Lebesgue measure is given by
	\begin{equation*}
		f(x)=\oneover{\sqrt{(2\pi)^n\det\Sigma}}\exp\Bigl(-\frac{1}{2}\big\langle x, \Sigma^{-1} x\big\rangle\Bigr),\qquad x\in\R^n.
	\end{equation*}
For \(\alpha,\theta>0\), we write \(X\sim\Gamma(\alpha,\vartheta) \) (resp. \(X\sim\beta(\alpha,\vartheta)\)) and say that \(X\) has a Gamma distribution (resp. a Beta distribution) with parameters \(\alpha\) and \(\vartheta \) if the probability density function of \(X\) w.r.t. to the Lebesgue measure is proportional to \(x\mapsto x^{\alpha-1}e^{-\vartheta x}\indic_{[0,\infty)}(x)\) (resp. \(x\mapsto x^{\alpha-1}(1-x)^{\vartheta-1}\indic_{[0,1]}(x)\)). We also say that \(X\) has a uniform distribution on \([0,1]\) if  \(X\sim\mathrm{Unif}([0,1])\coloneqq\beta(1,1)\) or an exponential distribution with parameter \(1\) if \(X\sim\exp(\vartheta)\coloneqq\Gamma(1,\vartheta)\).

The following properties of the aforementioned distributions are of interest and easy to verify by direct computation:
	\begin{align}
		&\text{if }X\sim\Gamma(\alpha,\vartheta) \text{ and }Y\sim\Gamma(\tilde\alpha,\vartheta)\text{ are independent, then }\frac{X}{X+Y}\sim\beta(\alpha,\tilde\alpha)\,,\label{eq:Gammaproperty1}\\
		&\text{if }X\sim\mathrm{Unif}([0,1]),\text{ then } X^k\sim\beta(1/k,1)\label{eq:Gammaproperty2}\,,
	\end{align}
for any \(\alpha,\tilde\alpha,\vartheta,k\in(0,\infty) \).

Given a real sequence \((a_n)_{n\in\N}\), we write \(a_n\equiv a\) if \(a_n=a\) for every \(n\in\N\). If \((b_n)_{n\in\N}\) is a positive sequence, we write 
\(a_n=\mathcal{O}(b_n)\) if there exists \(C\in(0,\infty) \) such that \(\abs{a_n}\le C b_n\) for every \(n\in\N\), and 
\(a_n=o(b_n)\) if \(\lim_{n\to\infty} (a_n/b_n)=0 \).

\subsection{The $\ell_p^n$-balls}\label{subsec22:LpBalls}

For \(n\in\N\), let \(x=(x_1,\ldots,x_n)\in\R^n\) and define the \(p\text{-norm}\) of $x$ via
	\begin{equation*}
	\label{eq:pnorm}
		\norm{x}_p\coloneqq
			\begin{dcases}
				\Bigl(\sum_{i=1}^n\abs{x_i}^p\Bigr)^{1/p}&\caseif p\in[1,\infty),\\
				\max_{1\leq i \leq n}\abs{x_i}&\caseif p=\infty.
			\end{dcases}
	\end{equation*}
The unit ball $\pball$ and sphere $\psphere$ with respect to this norm are defined as
	\begin{equation*}
		\pball\coloneqq\{x\in\R^n:\pnorm{x}\le 1\}\qquad\text{and}\qquad \psphere\coloneqq\{x\in\R^n:\pnorm{x}= 1\}=\partial\pball.
	\end{equation*}
As usual, we shall write \(\ell_p^n\) for the Banach space \((\R^n,\pnorm{\cdot})\). 
The exact value of \({\abs{\pball}}\) is known since Dirichlet \cite{D} and is given by
	\begin{equation*}
		\vol{\pball}=\frac{(2\Gamma(1+1/p))^n}{\Gamma(1+n/p)}.
	\end{equation*}
The interested reader may consult \cite{P} for a modern computation. The volume-normalized ball shall be denoted by $\mathbb D_p^n$ and is given by
	\begin{equation*}
	\mathbb{D}_p^n=\frac{\pball}{\abs{\pball}^{1/n}}.
	\end{equation*}
For convenience, in what follows we will use the convention that in the case \(p=\infty\), \(1/p\coloneqq0 \).
It is worth noticing that the restriction on the domain of \(p\) is due to the fact that an analogous definition of \(\pnorm{\,\!\cdot\,\!}\) for \(p<1\) does only result in a quasi-norm, meaning that the triangle inequality does not hold. As a consequence, \(\pball\) is convex if and only if \(p\ge1\). Although a priori many arguments of this  survey do not rely on \(\pnorm{\cdot}\) being a norm, we restrict our presentation to the case \(p\ge 1\), since it is necessary in some of the theorems.

\subsection{Grassmannian manifolds}\label{subsec23:Grassmannians}

The group of \((n\times n)\text{-orthogonal}\) matrices is denoted by $\mathbb O(n)$ and we let $\mathbb{SO}(n)$ be the subgroup of orthogonal $n\times n$ matrices with determinant $1$. As subsets of $\R^{n^2}$, $\mathbb{O}(n)$ and $\mathbb{SO}(n)$ can be equipped with the trace \(\sigma\text{-field}\) of $\borel(\R^{n^2})$. Moreover, both compact groups $\mathbb O(n)$ and $\mathbb{SO}(n)$ carry a unique Haar probability measure which we denote by $\eta$ and $\tilde{\eta}$, respectively. Since $\mathbb{O}(n)$ consists of two copies of $\mathbb{SO}(n)$, the measure $\eta$ can easily be derived from $\tilde{\eta}$ and vice versa.
Given $k\in\{0,1,\ldots,n\}$, we use the symbol $\G^n_k$ to denote the Grassmannian of $k$-dimensional linear subspaces of $\R^n$. We supply $\G_{k}^n$ with the metric \begin{equation*}
d(E,F)\coloneqq\max\Bigl\{\adjustlimits{\sup}_{x\in B_E}{\inf}_{y\in B_F} \norm{x-y}_2,\adjustlimits{\sup}_{y\in B_F}{\inf}_{x\in B_E} \norm{x-y}_2\Bigr\},\qquad  E,F\in \G_k^n,
\end{equation*} where $B_E$ and $B_F$ stand for the Euclidean unit balls in $E$ and $F$, respectively. The Borel $\sigma$-field on $\G_k^n$ induced by this metric is denoted by $\borel(\G_k^{n})$ and we supply the arising measurable space $\G_k^n$ with the unique Haar probability measure $\eta_k^n$. It can be identified with the image measure of the Haar probability measure $\tilde{\eta}$ on $\mathbb{SO}(n)$ under the mapping $\mathbb{SO}(n)\to\G_k^n,\, T\mapsto TE_0$ with $E_0\coloneqq\mathrm{span}(\{e_1,\ldots,e_k\})$. Here, we write $e_1\coloneqq(1,0,\ldots,0),e_2\coloneqq(0,1,0,\ldots,0),\ldots,e_n\coloneqq(0,\ldots,0,1)\in\R^n$ for the standard orthonormal basis in $\R^n$ and $\mathrm{span}(\{e_1,\ldots,e_k\})\in \G_k^n$, $k\in\{1,\ldots,n\}$, for the $k$-dimensional linear subspace spanned by the first $k$ vectors of this basis.

\subsection{Large deviation principles}\label{subsec24:LargeDeviations}

Consider a sequence \((X_n)_{n\in\N}\) of i.i.d integrable real random variables and let
\begin{equation*}
	S_n\coloneqq\oneover{n}\sum_{i=1}^n X_i
\end{equation*}
be the empirical average of the first \(n\) random variables of the sequence. It is well known that the law of large numbers provides the asymptotic behaviour of \(S_n\), as \(n\) tends to infinity. In particular, the strong law of large numbers says that
\begin{equation*}
	S_n\asconv\Ex[X_1].
\end{equation*}
If \(X_1\) has also positive and finite variance, then the classical central limit theorem states that the fluctuations of \(S_n\) around \(\Ex[X_1]\) are normal and of scale \(1/\sqrt{n}\). More precisely,
\begin{equation*}
	\sqrt{n}(S_n-\Ex[X_1])\dconv \mathcal{N}(0,\Var[X_1]).
\end{equation*}
One of the important features of the central limit theorem is its universality, i.e., that the limiting distribution is normal independently of the precise distribution of the summands \(X_1,X_2,\ldots\). This allows to have a good estimate for probabilities of the kind
\begin{equation*}
\Pr(S_n>x),\qquad x\in\R,
\end{equation*}
when \(n\) is large, but fixed.
However, such estimate can be quite imprecise if \(x\) is much larger than \(\Ex[X_1]\). Moreover, it does not provide any rate of convergence for such tail probabilities as \(n\) tends to infinity for fixed \(x\). 

In typical situations, if $S_n$ arises as a sum of $n$ independent random variables $X_1,\ldots,X_n$ with finite exponential moments, say, one has that
\begin{equation*}
\Pr(S_n>x)\approx e^{-n \rate(x)}, \qquad x>\Ex[X_1]
\end{equation*}
if $n\to\infty$, where $\rate$ is the so-called rate function. Here \(\approx\) expresses an asymptotic equivalence up to sub-exponential functions of \(n\). For concreteness, let us consider two examples. If \(\Pr(X_1=1)=\Pr(X_1=0)=1/2\), then
\begin{equation*}
\rate(x) =\begin{cases}
x\log x +(1-x)\log(1-x)+\log 2 & \caseif x\in[0,1],\\
+\infty &\other,
\end{cases}
\end{equation*}
which describes the upper large deviations.
If on the other hand $X_1\sim\mathcal N(0,\sigma^2)$, then the rate function is given by
\begin{equation*}
\rate(x) = \frac{x^2}{2\sigma^2}, \qquad x\in\R.
\end{equation*}
Contrarily to the universality shown in the central limit theorem, these two examples already underline that the function \(\rate\) and thus the decay of the tail probabilities is much more sensitive and specific to the distribution of \(X_1\). 

The study of the atypical situations (in contrast to the typical ones described in the laws of large numbers and the central limit theorem) is called Large Deviations Theory. The concept expressed heuristically in the examples above can be made formal in the following way. Let $\X\coloneqq(X_n)_{n\in\N}$ be a sequence of random vectors taking values in $\R^d$. Further, let $s\colon\N\to[0,\infty]$ be a non-negative sequence such that $s(n)\uparrow \infty$ and assume that  $\rate\colon\R^d\to[0,\infty]$ is a lower semi-continuous function, i.e., all of its lower level sets $\{x\in\R^d:\rate(x) \leq\ell \}$, $\ell\in[0,\infty]$, are closed. We say that \(\X\) satisfies a large deviation principle (or simply LDP) with speed $s(n)$ and rate function \(\rate\) if and only if
\begin{equation*}\label{eq:LDPdefinition}
-\inf_{x\in A^\circ}\rate(x) \leq\,\liminf_{n\to\infty}\;\!\oneover {s(n)}\log \Pr(X_n\in A)
\leq\limsup_{n\to\infty}\oneover {s(n)}\log\Pr(X_n\in A)\leq-\inf_{x\in\overline{A}}\rate(x)
\end{equation*}
for all $A\in\borel(\R^d)$. Moreover, \(\rate\) is said to be a good rate function if all of its lower level sets are compact. The latter property is essential to guarantee the so-called exponential tightness of the sequence of measures.

The following result, known as Cramér's Theorem, guarantees an LDP for the empirical average of a sequence of i.i.d. random vectors, provided that their common distribution is sufficiently nice (see, e.g. \cite[Theorem 27.5]{Kallenberg}).

\begin{theorem}[Cramér's Theorem]\label{thm:Cramér}
	Let \((X_n)_{n\in\N}\) be a sequence of i.i.d. random vectors in \(\R^d\) such that the cumulant generating function of \(X_1\),
	\begin{equation*}
	\Lambda(u)\coloneqq \log\Ex\bigl[\exp\scalar{X_1}{u}\bigr]\,,\qquad u\in\R^d,
	\end{equation*}
	is finite in a neighbourhood of \(0\in\R^d\). Let \(\mathbf{S}\coloneqq(\oneover{n}\sum_{i=1}^n X_i)_{n\in\N}\) be the sequence of the sample means. Then \(\mathbf{S}\) satisfies an LDP with speed \(n\) and good rate function \(\rate=\Lambda^*\), where
	\begin{equation*}
	\Lambda^*(x)\coloneqq\sup_{u\in\R^d}\bigl(\scalar{x}{u}-\Lambda(u) \bigr),\qquad x\in\R^d,
	\end{equation*}
	is the Fenchel-Legendre transform of \(\Lambda \). 
\end{theorem}

	Cramér's Theorem is a fundamental tool that allows to prove an LDP if the random variables of interest can be transformed into a sum of independent random variables.
	
	Sometimes there is the need to `transport' a large deviation principle from one space to another by means of a continuous function. This can be done with a device known as the contraction principle and we refer to \cite[Theorem 4.2.1]{DZ} or \cite[Theorem 27.11(i)]{Kallenberg}.

	\begin{proposition}[Contraction principle]\label{prop:contraction principle}
Let $d_1,d_2\in\N$ and let $F:\R^{d_1}\to\R^{d_2}$ be a continuous function. Further, let $\XX\coloneqq(X_n)_{n\in\N}$ be a sequence of $\R^{d_1}\text{-valued}$ random vectors that satisfies an LDP with speed $s(n)$ and rate function $\rate_\XX$. Then the sequence $\mathbf{Y}\coloneqq(F(X_n))_{n\in\N}$ of $\R^{d_2}\text{-valued}$ random vectors satisfies an LDP with the same speed and with good rate function $\rate_\mathbf{Y}=\rate_\XX\circ F^{-1}$, i.e., $\rate_\mathbf{Y}(y)\coloneqq\inf\{\rate_\XX(x):F(x)=y\}$, $y\in\R^{d_2}$, with the convention that $\rate_\mathbf{Y}(y)=+\infty$ if $F^{-1}(\{y\})=\emptyset$.
	\end{proposition}
	
While this form of the contraction principle is sufficient to analyse the large deviation behavior for $1$-dimensional random projections of $\ell_p^n$-balls, a refinement to treat the higher-dimensional cases is needed. To handle this situation, the classical contraction principle can be extended to allow a dependency on $n$ of the continuous function \(F\). We refer the interested reader to \cite[Corollary 4.2.21]{DZ} for the precise statement.

\section{Probability measures on convex bodies}\label{sec3:PMonConvexBodies}

There is a variety of probability measures that can be defined on the family of $\ell_p^n$-balls or spheres. We shall present some of them and their key properties below.

\subsection{Probability measures on an \(\ell_p^n\,\text{-ball}\)}\label{subsec31:PMonLpBalls}
	
One can endow \(\pball\) with a natural volume probability measure. This is defined as follows, 
	\begin{equation}
	\label{eq:puniform}
		\nu_p^n(A)\coloneqq\frac{\abs{A\cap\pball}}{\abs{\pball}},
	\end{equation}
	for any \(A\in\borel(\R^n)\).
	We also refer to \(\nu_p^n\) as the uniform distribution on \(\pball \). 
	
As far as \(\psphere\) is concerned, there are two probability measures that are of particular interest. The first is the so-called surface measure, which we denote by \(\sigma_p^n\), and which is defined as the normalised $(n-1)$-dimensional Hausdorff measure. The second, \(\mu_p^n\), is the so-called cone (probability) measure and is defined via
	\begin{equation}
	\label{eq:pconemeasure}
		\mu_p^n(A)\coloneqq\frac{\vol{[0,1]A}}{\vol{\pball}},\qquad A\in\borel(\psphere).
	\end{equation}
In other words, \(\mu_p^n(A)\) is the normalised volume of the cone that intersects \(\psphere \) in \(A\), intersected with \(\pball \). The cone measure is known to be the unique measure that satisfies the following polar integration formula for any integrable function \(f\) on \(\R^n\) (see, e.g., \cite[Proposition 1]{NR})
	\begin{equation}
	\label{eq:polarintegration}
		\int_{\R^n} f(x)\de x= n\,\vol{\pball}\int_0^\infty r^{n-1}\int_{\psphere} f(rz)\de\mu_p^n(z)\de r.
	\end{equation}
In particular, whenever \(f\) is \(p\text{-radial}\), i.e., there exists a function \(g\) defined on \(\R^+\) such that \(f(x)=g(\pnorm{x})\), then
	\begin{equation}
	\label{eq:pradial}
		\int_{\R^n}g(\pnorm{x})\de x= n\,\vol{\pball} \int_0^\infty r^{n-1}g(r)\de r.
	\end{equation}
The relation between \(\sigma_p^n\) and \(\mu_p^n\) has been deeply investigated. It is known, for example, that they coincide whenever \(p\in\{1,2,\infty\}\) (see, e.g., \cite{RR}). In the other cases, Naor \cite{N} provided a bound on the total variation distance of these two measures.
	\begin{proposition}
	Let \(\sigma_p^n\) and \(\mu_p^n\) be the surface probability and cone probability measure on \(\psphere \), respectively. Then
		\begin{align*}
			d_{\mathrm{TV}}(\sigma_p^n,\mu_p^n)&\coloneqq\sup\Big\{\abs{\sigma_p^n(A)-\mu_p^n(A)}:A\in\borel(\psphere)\Big\}
			\le C\Bigl(1-\oneover{p}\Bigr)\abs[\bigg]{1-\frac{2}{p}}\frac{\sqrt{np}}{n+p},
		\end{align*}
	where \(C\in(0,\infty)\) is an absolute constant.
	\end{proposition}
In particular, the above proposition ensures that for \(p\) fixed, such a distance decreases to \(0\) not slower than \(n^{-1/2}\).
	
An important feature of the cone measure is described by the following probabilistic representation, due to Schechtman and Zinn \cite{SZ1} (independently discovered by Rachev and R\"uschendorf \cite{RR}). We will below present a proof in a more general set-up.

	\begin{theorem}
	\label{thm:cone}
	Let \(n\in\N\) and \(p\in[1,\infty]\). Let \((Z_i)_{i\in\N}\) be independent and \(p\text{-generalized}\) Gaussian random variables, meaning absolutely continuous w.r.t. to the Lebesgue measure on $\R$ with density
		\begin{equation}
		\label{eq:pGauss}
			f_p(x)\coloneqq
				\begin{dcases}
					\oneover{2p^{1/p}\Gamma(1+1/p)}e^{-\abs{x}^p/p}&\caseif p\in[1,\infty)\,,\\
					\oneover{2}\indic_{[0,1]}(\abs{x})&\caseif p=\infty.
				\end{dcases}
		\end{equation}
	Consider the random vector \(Z\coloneqq(Z_1,\ldots,Z_n)\in\R^n\) and let $U\sim{\rm Unif}([0,1])$ be independent of $Z_1,\ldots,Z_n$. Then
		\begin{equation*}
			\frac{Z}{\pnorm{Z}}\sim\mu_p^n\qquad\text{and}\qquad U^{1/n}\frac{Z}{\pnorm{Z}}\sim\nu_p^n.
		\end{equation*}
	Moreover, \(Z/\pnorm{Z}\) is independent of \(\pnorm{Z}\).
	\end{theorem}
	
It is worth noticing that in \cite{SZ1} the density used by the authors for \(Z_1\) is actually proportional to \( x\mapsto\exp(-\abs{x}^p)\). As will become clear later, this difference is irrelevant as far as the conclusion of the theorem is concerned.

Indeed, although the statement of \Cref{thm:cone} reflects the focus of this survey on the \(\ell_p^n\text{-balls} \) and the literature on the topic, its result is not strictly dependent on the particular choice of \(f_p\) in \Cref{eq:pGauss}. In fact, it is not even a prerogative of the \(\ell_p^n\text{-balls} \), as subsequently explained in \Cref{prop:SZGeneralK}.

\subsection{The cone measure on a symmetric convex body}\label{subsec32:ConeMeasureSymmConvBody}
	
Consider a symmetric convex body \(K\subseteq\R^n\), meaning that if \(x\in K\) then also \(-x\in K\). Define the functional \(\norm{\cdot}_K\colon\R^n\to[0,\infty) \) by
	\begin{equation*}
		\norm{x}_K\coloneqq\inf\{r>0:x\in r K\}.
	\end{equation*}
 The functional \(\norm{\cdot}_K\) is known as the Minkowski functional associated with \(K\) and, under the aforementioned conditions on \(K\), defines a norm on \(\R^n \). We will also say that \(\norm{x}_K \) is the \(K\text{-norm}\) of the vector \(x\in\R^n\). Whenever a function on \(\R^n\) is dependent only on \(\norm{\cdot}_K \), we say that it is a \(K\text{-radial}\) function. Analogously, we call a probability measure \(\Kradial\)  when its distribution function is \(\Kradial\). We will also write \(\pradial \) meaning \(\pball\text{-radial} \). 
 
 In analogy with Equations \eqref{eq:puniform} and \eqref{eq:pconemeasure}, it is possible to define a uniform probability measure \(\nu_K\) on \(K\) and a cone measure \(\mu_K \) on \(\partial K \), respectively, as
 	\begin{equation*}
 		\nu_K(A)\coloneqq\frac{\vol{A\cap K}}{\vol{K}}\qquad\text{and}\qquad
 		\mu_K(B)\coloneqq\frac{\vol{[0,1]B}}{\vol{K}},
 	\end{equation*}
for any \(A\in\borel(\R^n)\) and \(B\in\borel(\partial K)\). 

Note that \(\mu_K\), as a ratio of volumes, is invariant under a simultaneous transformation of both the numerator and the denominator. In particular, for any \(I\in\borel(\R^+)\), such that \(\vol{I}>0 \), it holds
	\begin{equation}
	\label{eq:anyI}
		\mu_K(B)=\frac{\vol{IB}}{\vol{I\partial K}},
	\end{equation}
for any \(B\in\borel(\partial K)\) (note that \(K=[0,1]\partial K\)).
This fact will be used in the proof of the following generalization of \Cref{thm:cone} to arbitrary symmetric convex bodies.

\begin{proposition}\label{prop:SZGeneralK}
 		Let \(K\subseteq\R^n\) be a symmetric convex body. Suppose that there exists a continuous function  \(f\colon[0,\infty)\to[0,\infty)\) with the property \(\int_{\R^n} f(\norm{x}_K)\de x=1\)\, such that the distribution of a random vector \(Z\) on \(\R^n\) is given by
 			\begin{equation*}
 					\Pr(Z\in A)=\int_A f(\norm{x}_K)\de x,
 			\end{equation*}
 		for any \(A\in\borel(\R^n)\). Also, let $U\sim{\rm Unif}([0,1])$ be independent of $Z$. Then,
	\begin{equation}
	\label{eq:Kcone}
		\frac{Z}{\norm{Z}_K}\sim\mu_K\qquad\text{and}\qquad U^{1/n}\frac{Z}{\norm{Z}_K}\sim\nu_K.
	\end{equation}
In addition, \(Z/\norm{Z}_K\) is independent of \(\norm{Z}_K\). 
\end{proposition}

The proof of \Cref{prop:SZGeneralK} is based on the following polar integration formula, which generalizes \Cref{eq:polarintegration}. It says that for measurable functions $h:\R^n\to[0,\infty)$,
\begin{equation}\label{eq:PolarInterationGeneralK}
\int_{\R^n}h(x)\de x = n\vol{K}\int_0^{\infty}r^{n-1}\int_{\partial K}h(rz)\,\de\mu_K(z)\de r.
\end{equation}
By the usual measure-theoretic standard procedure to prove \Cref{eq:PolarInterationGeneralK} it is sufficient to consider functions $h$ of the form $h(x)={\bf 1}_A(x)$, where $A=(a,b)E$ with $0<a<b<\infty$ and $E$ a Borel subset of $\partial K$. However, in this case, the left-hand side is just $\vol{A}$, while for the right-hand side we obtain, by definition of the cone measure $\mu_K$,
\begin{align*}
n\vol{K}\int_0^\infty r^{n-1}\indic_{(a,b)}(r)\int_{\partial K}\indic_{E}(z)\,\de\mu_K(z)\de r = n\vol{K}\int_a^b r^{n-1}\de r\,\frac{\vol{[0,1]E}}{\vol{K}} = (b^n-a^n)\vol{[0,1]E},
\end{align*}
which is clearly also equal to $\vol{A}$.

\begin{proof}[Proof of \Cref{prop:SZGeneralK}]
Let $\phi:\R^n\to\R$ and $\psi:\R\to\R$ be non-negative measurable functions. Applying the polar integration formula, \Cref{eq:PolarInterationGeneralK}, yields
\begin{equation*}
\begin{split}
\Ex\Bigl[\phi\Bigl(\frac{Z}{\norm{Z}_K}\Bigr)\psi(\norm{Z}_K)\Bigr] &= \int_{\R^n}\phi\Bigl(\frac{x}{\norm{x}_K}\Bigr)\psi(\norm{x}_K)f(\norm{x}_K)\de x\\
&=n\vol{K}\int_0^\infty \psi(r)f(r)r^{n-1}\de r\,\int_{\partial K}\phi(z)\de\mu_K(z).
\end{split}
\end{equation*}
By the product structure of the last expression this first shows the independence of $Z/\norm{Z}_K$ and $\norm{Z}_K$. Moreover, choosing $\psi\equiv 1$ we see that
\begin{equation*}
\Ex \phi\Bigl(\frac{Z}{\norm{Z}_K}\Bigr) = n\vol{K}\int_0^\infty f(r)r^{n-1}\de r\,\int_{\partial K}\phi(z)\de\mu_K(z) = \int_{\partial K}\phi(z)\de\mu_K(z)
\end{equation*}
by definition of $f$. This proves that $Z/\norm{Z}_K\sim\mu_K$. That $U^{1/n}\frac{Z}{\norm{Z}_K}\sim\nu_K$ finally follows from the fact that  $U^{1/n}\sim\beta(n,1)$, which has density $r\mapsto nr^{n-1}$ for $r\in(0,1)$.
\end{proof}

The main reason why the theory treated in this survey is restricted to \(\ell_p^n\text{-balls} \), and not to more general convex bodies \(K\), is that \(\ell_p^n\text{-balls} \) are a class of convex bodies whose Minkowski functional is of the form
	\begin{equation}
	\label{eq:niceMinkowski}
		\norm{x}_K=F\Bigl(\sum_{i=1}^n f_i(x_i)\Bigr)
	\end{equation}
for certain functions \(f_1,\ldots, f_n\) and invertible positive function \(F\). This is necessary for \(Z\) to have independent coordinates. Indeed, in this case one can assign a joint density on \(Z\) that factorizes into its components, like for example (omitting the normalizing constant),
	\begin{equation*}
		e^{-F^{-1}(\norm{x}_K)}=e^{-\sum_{i=1}^n f_i(x_i)}=\prod_{i=1}^{n}e^{- f_i(x_i)},
	\end{equation*}
which ensures the independence of the coordinates \(Z_i\) of \(Z\).

Already for slightly more complicated convex bodies than \(\ell_p^n\text{-balls} \), \Cref{eq:niceMinkowski} no longer holds. For example, considering the convex body defined as
	\begin{equation*}
	 	\mathbb{B}^2_{1,2}\coloneqq\{x\in\R^2: \abs{x_1}+x_2^2\le 1\}.
	\end{equation*}
It can be computed that \(\norm{x}_{\mathbb{B}^2_{1,2}}=\abs{x_1}/2+\sqrt{x_1^2/4+x_2^2}\), which is not of the form \eqref{eq:niceMinkowski}.

On the other hand, the coordinate-wise representation of the density of \(Z\) in the precise form given by \Cref{eq:pGauss}, is also convenient to explicitly compute the distribution of some functionals of \(Z\), as we will see in the following section.

\subsection{A different probabilistic representation for \(p\text{-radial}\) probability measures}

Another probabilistic representation for a \(p\text{-symmetric}\) probability measure on \(\pball\) has been given by Barthe, Gu\'edon, Mendelson and Naor \cite{BGMN} in the following way,
	\begin{theorem}
	\label{thm:W}
	
		Let \(Z\) be a random vector in \(\R^n\) defined as in \(\Cref{thm:cone}\). Let \(W\) be a non-negative random variable with probability distribution \(\Pr_W\) and independent of \(Z\). Then
		\begin{equation*}
			\frac{Z}{(\pnorm{Z}^p+W)^{1/p}}\sim \Pr_W(\{0\})\,\mu_p^n+\mathrm{H}_W(\cdot)\,\nu_p^n,
		\end{equation*}
		where \(\mathrm{H}_W\colon\pball\to\R\), \(\mathrm{H}_W(x)=h(\pnorm{x})\), with
		\begin{equation*}
			h(r)=\oneover{\Gamma(1+n/p)(1-r^p)^{1+n/p}}\int_{(0,\infty)} s^{n/p}e^{s r^p/(r^p-1)}\de\Pr_W(s).
		\end{equation*}
	\end{theorem}
	\begin{remark}
Note that all the distributions obtainable from \Cref{thm:W} are \(p\text{-radial}\), especially the \(p\text{-norm}\) of \(Z/(\pnorm{Z}^p+W)^{1/p}\) is
	\begin{equation*}
		R=\Bigl(\frac{\pnorm{Z}^p}{\pnorm{Z}^p+W}\Bigr)^{1/p}.
	\end{equation*}
	 Moreover, some particular choices of \(W\) in \Cref{thm:W} lead to interesting distributions:
		\begin{enumerate}[label=(\roman*)]
			\item When \(W\equiv 0\) we recover the cone measure of \Cref{thm:cone};
			\item For \(\alpha>0\), choosing \(W\sim\Gamma(\alpha,1)\) results in the density proportional to \(x\mapsto (1-\pnorm{x}^p)^{\alpha-1}\) for \(\pnorm{x}\le1 \).
			\item As a particular case of the previous one, when \(W\sim\exp(1)=\Gamma(1,1) \), then \(H_W\equiv 1\) and
			\begin{equation*}	
				\frac{Z}{(\pnorm{Z}^p+W)^{1/p}}\sim\nu_p^n.
			\end{equation*}
		This is not in contrast with \Cref{thm:cone}. Indeed, it is easy to compute that
			\begin{equation*}
				\pnorm{Z}^p\sim\Gamma(n/p,1).
			\end{equation*}
		 In view of the properties \eqref{eq:Gammaproperty1} and \eqref{eq:Gammaproperty2}, this implies
			\begin{equation*}
				\frac{\pnorm{Z}^p}{\pnorm{Z}^p+W}\sim\beta(n/p,1)\sim U^{p/n}.
			\end{equation*}  
			As a consequence of this fact, the orthogonal projection of the cone measure \(\mu_p^{n+p}\) on \(\partial \mathbb{B}_p^{n+p}\) onto the first n coordinates is \(\nu_p^n \). Indeed, if \(W=\sum_{i=n+1}^{n+p}\abs{Z_i}^p\), then \(W\sim\exp(1)\), while
			\begin{equation*}
				\frac{Z}{(\pnorm{Z}^p+W)^{1/p}}=\frac{(Z_1,\ldots,Z_n)}{(\sum_{i=1}^{n+p}\abs{Z_i}^p)^{1/p}}
			\end{equation*}
		is the required projection. We refer to \cite[Corollaries 3-4]{BGMN} for more details in this direction.
		\end{enumerate}
	\end{remark}
	
\section{Central limit theorems \& Laws of large numbers}\label{sec4:CLTandLLN}

The law of large numbers and the central limit theorem are arguably among the most prominent limit theorems in probability theory. Thanks to the probabilistic representation for the various geometric measures on $\ell_p^n$-balls described in Section \ref{subsec31:PMonLpBalls}, both of these limit theorems can successfully applied to deduce information about the geometry of $\ell_p^n$-balls. This -- by now classical -- approach will be described here, but we will also consider some more recent developments in this direction as well as several generalizations of known results.

\subsection{Classical results: Limit theorems \`a la Schechtman-Schmuckenschl\"ager}\label{subsec41:ClassiclSchechtSchmuck}


The following result on the absolute moments of a \(p\text{-generalized}\) Gaussian random variable is easy to derive by direct computation, and therefore we omit its proof, which the reader can find in \cite[Lemma 4.1]{KPT}
\begin{lemma}\label{lem:MOMENTSpGenGAUSSIAN}
	Let \(p\in(0,\infty]\) and let \(Z_0\) be a \(p\text{-generalized}\) Gaussian random variable (i.e., its density is given by \Cref{eq:pGauss}). Then, for any \(q\in[0,\infty]\),
		\begin{equation*}
			\Ex\bigl[\abs{Z_0}^q\bigr]=
				\begin{dcases}
					\frac{p^{q/p}}{q+1}\frac{\Gamma\bigl(1+\frac{q+1}{p} \bigr)}{\Gamma\bigl(1+\frac{1}{p} \bigr)}=: M_p(q)&\caseif p<\infty, \\
					\oneover{q+1}=: M_\infty(q)&\caseif p=\infty.
				\end{dcases}
		\end{equation*}
\end{lemma}
For convenience, we will also indicate \(m_{p,q}\coloneqq M_p(q)^{1/q}\) and 
	\begin{equation*}
C_p(q,r)\coloneqq\Cov(\abs{Z_0}^q,\abs{Z_0}^r)=M_p(q+r)-M_p(q)M_p(r).
	\end{equation*}
We use the convention that \(M_\infty(\infty)=C_\infty(\infty,\infty)=C_\infty(\infty,q)=0\). The next theorem is a version of the central limit theorem in \cite[Proposition 2.4]{S}.

\begin{theorem}
	\label{thm:clt-qnorm-1dim}
	Let \(0<p,q<\infty\), \(p\neq q\) and \(X\sim\nu_p^n\). Then
		\begin{equation*}
			\sqrt n\Bigl(n^{1/p-1/q}\frac{\norm{X}_q}{m_{p,q}}-1 \Bigr)\dconv N,
		\end{equation*}
	where \(N\sim\mathcal{N}\bigl(0,\sigma^2_{p,q}\bigr)\) and 
		\begin{equation*}
		\label{eq:sigma}
		\begin{split}
			\sigma^2_{p,q}&\coloneqq\frac{C_p(q,q)}{q^2 M_p(q)^2}-\frac{2C_p(p,q)}{pqM_p(q)}+\frac{C_p(p,p)}{p^2}
		\end{split}
		\end{equation*}
\end{theorem}
	Note that, since \(M_p(p)=1\), then \( \sigma^2_{p,p}=0\). In fact, in such a case
		\begin{equation*}
			\sqrt{n}(\pnorm{X}-1)\dconv 0,
		\end{equation*}
	and a different normalization than \(\sqrt{n} \) is needed to obtain a non-degenerate limit distribution. Moreover, \(\sigma_{p,q}^2>0 \) whenever \(p\neq q\).
	
	For our purposes, it is convenient to define the following quantities
		\begin{equation*}
			k_{p,n}\coloneqq n^{1/p}\vol{\pball}^{1/n}, \qquad 	k_{q,n}\coloneqq n^{1/q}\vol{\mathbb{B}_q^n}^{1/n}
		\end{equation*}
	and
		\begin{equation*}
			A_{p,q,n}\coloneqq\frac{k_{p,n}}{m_{p,q}k_{q,n}}.
		\end{equation*}
	It is easy to verify with Sterling's approximation that, for any \(p,q>0\), \(A_{p,q,n}=A_{p,q}+\mathcal{O}(1/n)\) for \( A_{p,q}\in(0,\infty)\), as \(n\to\infty\).
	
	With this definition in mind, we exploit \Cref{thm:clt-qnorm-1dim} to prove a result on the volume of the intersection of \(\ell_p^n\text{-balls}\). This can be regarded as a generalization of the main results in Schechtman and Schmuckenschl\"ager \cite{SS1991}, and Schmuckenschl\"ager \cite{S1998,S}.
	
	\begin{corollary}
		\label{cor:intersectballs}
	Let \(0<p,q<\infty\) and \(p\neq q\). Let \(r\in[0,1]\) and \((t_n)_{n\in\N}\subseteq\R^+\) be such that
	\begin{equation*}
	\lim_{n\to\infty}\sqrt{n}(t_nA_{p,q}-1)=\Phi_{p,q}^{-1}(r),
	\end{equation*}
	where \(\Phi_{p,q}:[-\infty,+\infty]\to[0,1] \) is the distribution function of \(N\sim\mathcal{N}(0,\sigma^2_{p,q})\) and \(\sigma^2_{p,q}\) is defined in \Cref{eq:sigma},
	i.e.,
	\begin{equation*}
	\Phi_{p,q}(x)\coloneqq  \oneover{\sqrt{2\pi\sigma_{p,q}^2}}\int_{-\infty}^xe^{-s^2/(2\sigma_{p,q}^2)}\de s.
	\end{equation*}
	Then
	\begin{equation*}
	\lim_{n\to\infty}\vol[\big]{\mathbb{D}_p^n\cap t_n\mathbb{D}_q^n}=r.
	\end{equation*}
	In particular, when \(t_n\equiv t\), then
	\begin{equation*}
	\lim_{n\to\infty}\vol[\big]{\mathbb{D}_p^n\cap t\,\mathbb{D}_q^n}=
		\begin{cases}
			0&\caseif t<1/A_{p,q},\\
			1/2&\caseif t=1/A_{p,q},\\
			1&\caseif t> 1/A_{p,q}.
		\end{cases}
	\end{equation*}
\end{corollary}
\begin{proof}
	First of all, note that, since \(A_{p,q,n}=A_{p,q}+\mathcal{O}(1/n)\), then
	\begin{equation*}
	\lim_{n\to\infty}\sqrt{n}(t_nA_{p,q,n}-1)=\lim_{n\to\infty}\sqrt{n}(t_nA_{p,q}-1),
	\end{equation*}
	provided that the latter exists in \([-\infty,\infty]\), as per assumption.	In particular, taking the limit on both sides of the following equality,
	\begin{equation*}
	\Pr\bigl(\norm{X}_q\le t_n k_{p,n}k_{q,n}^{-1}n^{1/p-1/q} \bigr)=\Pr\bigl(\sqrt{n}(n^{1/p-1/q}m_{p,q}^{-1}\norm{X}_q-1)\le \sqrt{n}(t_n A_{p,q,n}-1) \bigr),
	\end{equation*}
	we get, because of \Cref{thm:clt-qnorm-1dim},
	\begin{equation*}
	\lim_{n\to\infty}\Pr\bigl(\norm{X}_q\le t_n k_{p,n}k_{q,n}^{-1}n^{1/p-1/q} \bigr)
	=\Pr\bigl(N\le \Phi_{p,q}^{-1}(r)\bigr)
	=r.
	\end{equation*}
	On the other hand, it is true that the following chain of equalities hold:
	\begin{equation*}
	\begin{split}
	\Pr\bigl(\norm{X}_q\le t_n k_{p,n}k_{q,n}^{-1}n^{1/p-1/q} \bigr)
	&=\frac{\vol{z\in\pball:z\in t_n k_{p,n}k_{q,n}^{-1}n^{1/p-1/q}\mathbb{B}_q^n}}{\vol{\pball}}\\
	&=\vol[\big]{z\in\vol{\pball}^{-1/n}\pball:z\in t_n k_{p,n}k_{q,n}^{-1}n^{1/p-1/q}\vol{\pball}^{-1/n}\mathbb{B}_q^n}\\
	&=\vol{z\in \mathbb{D}_p^n:z\in t_n\mathbb{D}_q^n}\\
	&=\vol{ \mathbb{D}_p^n\cap t_n \mathbb{D}_q^n},
	\end{split}
	\end{equation*}
	which concludes the main part of proof. For the last observation, note that for any \(t\) constant, either \(\sqrt{n}(t A_{p,q}-1)\equiv 0\) or it diverges.
\end{proof}

\subsection{Recent developments}\label{subsec42:RecentDevelopmentsMultivariateCLT+OutlookMatrix}

\subsubsection{The multivariate CLT}

We present here a multivariate central limit theorem that recently appeared in \cite{KPT}. It constitutes the multivariate generalization of \Cref{thm:clt-qnorm-1dim}. Similar to the classical results of Schechtman and Schmuckenschl\"ager \cite{SS1991}, and Schmuckenschl\"ager \cite{S1998,S} this was used to study intersections of (this time multiple) $\ell_p^n$-balls. In part \(1.\), we replace the original assumption \(X\sim\nu_p^n\) of \cite{KPT} to a more general one, that appears naturally from the proof. Part \(2.\) is substantially different and cannot be generalized with the same assumption.

\begin{theorem}
	\label{thm:multCLT}
	Let \(n,k\in\N\) and \(p\in[1,\infty]\). 
	\begin{enumerate}
		\item Let \(X\) be a continuous \(p\text{-radial}\) random vector in \(\R^n\) such that
		\begin{equation}
		\label{eq:Rconv}
		\sqrt{n}\big(1-\pnorm{X}\big)\pconv 0.
		\end{equation}
		 Fix a \(k\text{-tuple}\) \((q_1,\ldots,q_k)\in([1,\infty)\setminus\{p\})^k\). 
		We have the multivariate central limit theorem
			\begin{equation*}
				\sqrt{n}\Bigl(n^{1/p-1/q_1}\frac{\norm{X}_{q_1}}{m_{p,q_1}}-1,\ldots, n^{1/p-1/q_k}\frac{\norm{X}_{q_k}}{m_{p,q_k}}-1\Bigr)\dconv N,
			\end{equation*}
			where \(N=(N_1,\ldots,N_k)\sim\mathcal{N}(0,\Sigma)\), with covariance matrix \(\Sigma=(c_{i,j})_{i,j=1}^k \) whose entries are given by
			\begin{equation}
			\label{eq:multisigma}
				c_{i,j}\coloneqq
					\begin{dcases}
				\oneover{q_iq_j}\biggl( \frac{\Gamma(\frac{1}{p})\Gamma(\frac{q_i+q_j+1}{p})}{\Gamma(\frac{q_i+1}{p})\Gamma(\frac{q_j+1}{p})}-1\biggr)-\oneover{p}&\caseif p<\infty,\\
				\oneover{q_i+q_j+1}&\caseif p=\infty.
				\end{dcases}
			\end{equation}
		\item Let \(X\sim\nu_p^n\). If \(p<\infty\), then we have the non-central limit theorem
			\begin{equation*}
				\frac{n^{1/p}}{(p\log n)^{1/p-1}}\norm{X}_\infty-A_n^{(p)}\dconv G,
			\end{equation*}
		where 
			\begin{equation*}
				A_n^{(p)}\coloneqq p\log n -\frac{1-p}{p}\log (p\log n)+\log(p^{1/p}\Gamma(1+1/p))
			\end{equation*}
		and \(G\) is a Gumbel random variable with distribution function \(\R\ni t\mapsto e^{-e^{-t}}\).
	\end{enumerate}
\end{theorem}
\begin{remark}
Note that the assumptions of \Cref{thm:multCLT} include the cases \(X\sim\nu_p^n\)  and \( X\sim\mu_p^n\). In fact, condition \eqref{eq:Rconv} is just the quantitative version of the following concept: to have Gaussian fluctuations it is necessary that the bigger \(n\) gets, the more the distribution of \(X\) is concentrated in near \(\partial\pball\). It is relevant to note that \eqref{eq:Rconv}  also keeps open the possibility for a non-trivial limit distribution when rescaling \((1-\pnorm{X})\) with a sequence that grows faster than \(\sqrt{n}\). This would yield a limit-theorem for \(\pnorm{X}\). For example, when \(X\sim\nu_p^n\), we already noted that \(\pnorm{X}\disteq U^{1/n}\), so that
\begin{equation*}
	n(1-\pnorm{X})\dconv E\sim\exp(1).
\end{equation*}
On the other hand, when \(X\sim\mu_p^n\), then \(1-\pnorm{X}\equiv 0 \) .
\end{remark}
\begin{proof}
	We only give a proof for the first part of the theorem, the second one can be found in \cite{KPT}.
	
	Let first \(p\in[1,\infty)\). Consider a sequence of independent \(p\text{-generalized}\) Gaussian random variables \((Z_j)_{j\in\N}\), also independent from every \(X\). Set \( Z=(Z_1,\ldots,Z_n)\). For any \(n\in \N\) and \(i\in\{1,\ldots, k\}\), consider the random variables
		\begin{equation*}
			\xi_n^{(i)}\coloneqq\oneover{\sqrt{n}}\sum_{j=1}^n\bigl(\abs{Z_j}^{q_i}-M_p(q_i)\bigr)\qquad\text{and}\qquad\eta_n\coloneqq\oneover{\sqrt{n}}\sum_{j=1}^n\bigl(\abs{Z_j}^{p}-1\bigr).
		\end{equation*}
		According to the classical multivariate central limit theorem, we get
		\begin{equation*}
			(\xi_n^{(1)},\ldots,\xi_n^{(k)},\eta_n)\dconv 	(\xi^{(1)},\ldots,\xi^{(k)},\eta)\sim\mathcal{N}(0,\widetilde\Sigma)
		\end{equation*}
		with covariance matrix given by
		\begin{equation*}
			\widetilde\Sigma=\begin{pmatrix}
			C_p(q_1,q_1) & \cdots & C_p(q_1,q_k) & C_p(q_1,p) \\
			\vdots & \!\ddots & \vdots & \vdots \\
				C_p(q_k,q_1) & \cdots & C_p(q_k,q_k) & C_p(q_k,p) \\
				\vspace{-10pt} \\
				C_p(p,q_1) & \cdots & C_p(p,q_k) & C_p(p,p)
			\end{pmatrix}
		\end{equation*}
	Using \Cref{thm:cone} and the aforementioned definitions we can write, for \(i\in\{1,
	\ldots,k\} \),
		\begin{equation*}
		\begin{split}
			\norm{X}_{q_i}&\disteq \frac{\pnorm{X}\norm{Z}_{q_i}}{\pnorm{Z}}\\
				&=\pnorm{X}\frac{(nM_p(q_i)+\sqrt{n}\xi_n^{(i)})^{1/q_i}}{(n+\sqrt{n}\eta_n)^{1/p}}\\
				&=\pnorm{X}\frac{(nM_p(q_i))^{1/q_i}}{n^{1/p}}F_i\Bigl( \frac{\xi_n^{(i)}}{\sqrt{n}},\frac{\eta_n}{\sqrt{n}}\Bigr)\\
				&=\pnorm{X}\,n^{1/q_i-1/p}m_{p,q}F_i\Bigl( \frac{\xi_n^{(i)}}{\sqrt{n}},\frac{\eta_n}{\sqrt{n}}\Bigr)\\
				&=(\pnorm{X}-1)n^{1/q_i-1/p}m_{p,q}F_i\Bigl( \frac{\xi_n^{(i)}}{\sqrt{n}},\frac{\eta_n}{\sqrt{n}}\Bigr)+n^{1/q_i-1/p}m_{p,q}F_i\Bigl( \frac{\xi_n^{(i)}}{\sqrt{n}},\frac{\eta_n}{\sqrt{n}}\Bigr)
		\end{split}
		\end{equation*}
	where we defined the function \(F_i\colon\R\times (\R\setminus \{-1\} )\to\R \) as
	\begin{equation*}
		F_i(x,y)\coloneqq\frac{(1+x/M_p(q_i))^{1/q_i}}{(1+y)^{1/p}}.
	\end{equation*}
	Note that \(F_i\) is continuously differentiable around \((0,0)\) with Taylor expansion given by
		\begin{equation*}
			F_i(x,y)=1+\frac{x}{q_iM_p(q_i)}-\frac{y}{p}+\mathcal{O}(x^2+y^2).
		\end{equation*}
	Since, for the law of large numbers, \(\xi_n^{(i)}/\sqrt{n}\asconv 0 \) and \(\eta_n/\sqrt{n}\asconv 0 \), the previous equation means that there exists a random variable \(C\), independent of \(n\), such that 
	\begin{equation*}
		\abs[\Big]{F_i\Bigl( \frac{\xi_n^{(i)}}{\sqrt{n}},\frac{\eta_n}{\sqrt{n}}\Bigr)-\Bigl(1+\frac{1}{q_iM_p(q_i)}\frac{\xi_n^{(i)}}{\sqrt{n}}-\frac{1}{p}\frac{\eta_n}{\sqrt{n}}\Bigr)}\le C\frac{(\xi_n^{(i)})^2+\eta_n^2}{n}.
	\end{equation*}
In particular,
	\begin{equation*}
	\begin{split}
	 \sqrt{n}(&\pnorm{X}-1)\Bigl(1+\frac{1}{q_iM_p(q_i)}\frac{\xi_n^{(i)}}{\sqrt{n}}-\frac{1}{p}\frac{\eta_n}{\sqrt{n}}-C\frac{(\xi_n^{(i)})^2+\eta_n^2}{n}\Bigr)\\
	 &+\Bigl(\frac{1}{q_iM_p(q_i)}\xi_n^{(i)}-\frac{1}{p}\eta_n-C\frac{(\xi_n^{(i)})^2+\eta_n^2}{\sqrt{n}}\Bigr)\\
			&\qquad\qquad\le\sqrt{n}\Bigl(n^{1/p-1/q_i}\frac{\norm{X}_{q_i}}{m_{p,q_i}}-1\Bigr)\\
			&\qquad\qquad\qquad\qquad
			\le\sqrt{n}(\pnorm{X}-1)\Bigl(1+\frac{1}{q_iM_p(q_i)}\frac{\xi_n^{(i)}}{\sqrt{n}}-\frac{1}{p}\frac{\eta_n}{\sqrt{n}}+C\frac{(\xi_n^{(i)})^2+\eta_n^2}{n}\Bigr)\\
			&\quad\qquad\qquad\qquad\qquad\qquad+\Bigl(\frac{1}{q_iM_p(q_i)}\xi_n^{(i)}-\frac{1}{p}\eta_n+C\frac{(\xi_n^{(i)})^2+\eta_n^2}{\sqrt{n}}\Bigr)
	\end{split}
	\end{equation*}
	Note that the first summand of both bounding expressions tends to \(0\) in distribution by assumption \eqref{eq:Rconv}, while the second converges in distribution to \(\frac{1}{q_iM_p(q_i)}\xi^{(i)}-\frac{1}{p}\eta\). This implies that 
	\begin{equation*}
		\sqrt{n}\Bigl(n^{1/p-1/q_i}\frac{\norm{X}_{q_i}}{m_{p,q_i}}-1\Bigr)\dconv \frac{1}{q_iM_p(q_i)}\xi^{(i)}-\frac{1}{p}\eta\eqqcolon N_i,
	\end{equation*}
	where \(N_i\) is a centered Gaussian random variable.
	To obtain the final multivariate central limit theorem, we only have to compute the covariance matrix \( \Sigma\). For \(\{i,j\}\subseteq\{1,\ldots,k \} \), its entries are given by
		\begin{equation*}
			\begin{split}
				c_{i,j}&=\Cov\Bigl(\frac{\xi^{(i)}}{q_iM_p(q_i)}-\frac{\eta}{p}, \frac{\xi^{(j)}}{q_jM_p(q_j)}-\frac{\eta}{p}\Bigr)\\
				&=\frac{\Cov(\xi^{(i)},\xi^{(j)})}{q_i q_jM_p(q_i)M_p(q_j)}-\frac{1}{p}\Bigl(\frac{\Cov(\xi^{(i)},\eta)}{q_iM_p(q_i)}+\frac{\Cov(\eta,\xi^{(j)})}{q_jM_p(q_j)}\Bigr)+\frac{\Cov(\eta,\eta)}{p^2}\\
				&=\frac{C_p(q_i,q_j)}{q_i q_jM_p(q_i)M_p(q_j)}-\frac{1}{p}\Bigl(\frac{C_p(q_i,p)}{q_iM_p(q_i)}+\frac{C_p(q_j,p)}{q_jM_p(q_j)}\Bigr)+\frac{C_p(p,p)}{p^2},
			\end{split}
		\end{equation*}
	and this can be made explicit to get \Cref{eq:multisigma}.
	The remaining case of \(p=\infty\) can be repeated using the aforementioned conventions on the quantities \(M_\infty\) and \(C_\infty\).
\end{proof}
\begin{remark}
	From the proof is evident that in the case when \(\sqrt{n}(\pnorm{X}-1)\) converges in distribution to a random variable \(F\), independence yields, for every \(i\in\{1,\ldots,k\}\), the convergence in distribution
	\begin{equation*}
		\sqrt{n}\Bigl(n^{1/p-1/q_i}\frac{\norm{X}_{q_i}}{m_{p,q_i}}-1\Bigr)\dconv F+ N_i
	\end{equation*}
in which case the limiting random variable is not normal in general. Analogously, if there exists a sequence \((a_n)_{n\in\N}\),  \(a_n=o(\sqrt{n})\) and a random variable \(F\) such that 
	\begin{equation*}
	a_n(\pnorm{X}-1)\dconv F,
	\end{equation*}
	then the previous proof, with just a change of normalization, yields the limit theorem
	\begin{equation*}
			a_n\Bigl(n^{1/p-1/q}\frac{\norm{X}_{q}}{m_{p,q}}-1\Bigr)\dconv F
	\end{equation*}
	for every \(q\in[1,\infty)\), as $n\to\infty$.
\end{remark}

In analogy to \Cref{cor:intersectballs}, one can prove in a similar way the following result concerning the simultaneous intersection of several dilated \(\ell_p\text{-balls} \). In particular, we emphasize that the volume of the simultaneous intersection of three balls $\mathbb{D}_p^n\cap t_1\mathbb{D}_{q_1}^n\cap t_2\mathbb{D}_{q_2}$ is \textit{not} equal to $1/4$ if these balls are in `critical' position, as one might conjecture in view of \Cref{cor:intersectballs}.

\begin{corollary}
		Let \(n,k\in\N\) and \(p\in[1,\infty]\). Fix a \(k\text{-uple}\) \((q_1,\ldots,q_k)\in([1,\infty)\setminus\{p\})^k\) . Let \(t_1,
		\ldots,t_k\) be positive constants and define the sets \(I_\star\coloneqq\{i\in\{1,\ldots,k\}:A_{p,q_i}t_i\star 1\}\), where \(\star\) is one of the symbols \(>\), \(=\) or \(<\). Then,
		\begin{equation*}
\lim_{n\to\infty}\vol{\mathbb{D}_p^n\cap t_1\mathbb{D}_{q_1}^n\cap\cdots\cap t_k\mathbb{D}_{q_k}^n}=	
	\begin{cases}
			1								   &\caseif	\# I_>=k,\\
			\Pr(N_i\le 0:i\in I_=)	&\caseif    \# I_=\ge 1\text{ and } \# I_<=0,\\
			0								  &\caseif	    \# I_<\ge 1,
	\end{cases}
			\end{equation*}
			where \(N=(N_1,\ldots,N_k)\) is as 
			in \Cref{thm:multCLT}.
\end{corollary}

\subsubsection{Outlook -- the non-commutative setting}\label{subsubsec: outlook non-commutative schechtman-schmuckenschlaeger}

Very recently, Kabluchko, Prochno and Th\"ale obtained in \cite{KPT2018} a non-commutative analogue of the classical result by Schechtman and Schmuckenschl\"ager \cite{SS1991}. Instead of considering the family of $\ell_p^n$-balls, they studied the volumetric properties of unit balls in classes of classical matrix ensembles. 

More precisely, we let $\beta\in\{1,2,4\}$ and consider the collection $\mathscr H_n(\mathbb{F}_\beta)$ of all self-adjoint $n\times n$ matrices with entries from the (skew) field $\mathbb{F}_\beta$, where $\mathbb{F}_1=\R$, $\mathbb{F}_2=\mathbb C$ or $\mathbb{F}_4=\mathbb H$ (the set of Hamiltonian quaternions). By $\lambda_1(A),\ldots,\lambda_n(A)$ we denote the (real) eigenvalues of a matrix $A$ from $\mathscr H_n(\mathbb{F}_\beta)$ and consider the following matrix analogues of the classical $\ell_p^n$-balls discussed above:
\begin{equation*}
\mathbb{B}_{p,\beta}^n\coloneqq\Bigl\{A\in \mathscr H_n(\mathbb{F}_\beta):\sum_{j=1}^n\abs{\lambda_j(A)}^p \leq 1\Bigr\},\qquad \beta \in\{1,2,4 \}\quad\text{and}\quad p\in[1,\infty],
\end{equation*}
where we interpret the sum in brackets as $\max\{\lambda_j(A):j=1,\ldots,n\}$ if $p=\infty$. As in the case of the classical $\ell_p^n$-balls  we denote by $\mathbb D_{p,\beta}^n$, $\beta\in\{1,2,4\}$ the volume normalized versions of these matrix unit balls. Here the volume can be identified with the $(\beta\frac{n(n-1)}{2}+\beta n)$-dimensional Hausdorff measure on $\mathscr H_n(\mathbb{F}_\beta)$.

\begin{theorem}\label{thm:ApplInto}
Let $1\leq p, q <\infty$ with $p\neq q$ and $\beta\in\{1,2,4\}$. Then
\[
\lim_{n\to\infty}\vol{\mathbb D^n_{p,\beta}\cap t\, \mathbb D^n_{q,\beta}}=
\begin{cases}
0 &\caseif t < e^{\frac{1}{2p} - \frac{1}{2q}} \bigl(\frac{2p}{p+q}\bigr)^{1/q},\\
1 &\caseif t > e^{\frac{1}{2p} - \frac{1}{2q}} \bigl(\frac{2p}{p+q}\bigr)^{1/q}\,.
\end{cases}
\]
\end{theorem} 

To obtain this result, one first needs to study the asymptotic volume of the unit balls of $\mathscr H_n(\mathbb{F}_\beta)$. This is done by resorting to ideas from the theory of logarithmic potentials with external fields. The second ingredient is a weak law of large numbers for the eigenvalues of a matrix chosen uniformly at random from $\mathbb B_{p,\beta}^n$. For details we refer the interested reader to \cite{KPT2018}.

\section{Large deviations vs.\ large deviation principles}\label{sec5:LargeDeviations}

The final section is devoted to large deviations and large deviation principles for geometric characteristics of $\ell_p^n$-balls. We start by presenting some classical results on large deviations related to the geometry of $\ell_p^n$-balls due to Schechtman and Zinn. Its LDP counterpart has entered the stage of asymptotic geometry analysis only recently in \cite{KPT}. We then continue by presenting a large deviation principle for $1$-dimensional random projections of $\ell_p^n$-balls of Gantert, Kim and Ramanan \cite{GKR2017}. Finally, we present a similar result for higher-dimensional projections as well. 

\subsection{Classical results: Large deviations \`a la Schechtman-Zinn}\label{subsec51:ClassicalConcentrationIneq}


We start by rephrasing the large deviation inequality of Schechtman and Zinn \cite{SZ1}. It is concerned with the $\ell_q$-norm of a random vector in an $\ell_p^n$-balls. The proof that we present follows the argument of \cite{N}.

\begin{theorem}\label{thm:SZInequality}
	Let \(1\le p< q\le\infty\) and \(X\sim\nu_p^n\) or \(X\sim\mu_p^n\). Then there exists a constant \(c\in(0,\infty)\), depending only on \(p\) and \(q\), such that
	\begin{equation*}
		\Pr(n^{1/p-1/q}\norm{X}_q>z)\le\exp(-c\,n^{p/q} z^p),
	\end{equation*}
	for every \(z>1/c\).
\end{theorem}
\begin{proof}
We sketch the proof for the case that $X\sim\mu_p^n$. Let \(Z_1,\ldots,Z_n\) be \(p\text{-generalized}\) Gaussian random variables and put $S_r\coloneqq\abs{Z_1}^r+\ldots+\abs{Z_n}^r$ for $r\geq 1$. Now observe that by the exponential Markov inequality and \Cref{thm:cone}, for $t>0$,
\begin{align*}
\Pr(n^{1/p-1/q}\norm{X}_q>z) &=\Pr\Bigl(\frac{S_q^{p/q}}{S_p}>\frac{z^p}{n^{1-p/q}}\Bigr)\\ &\leq\exp\Bigl(-\frac{tz^p}{n^{1-p/q}}\Bigr)\Ex\exp\Bigl(t\frac{S_q^{p/q}}{S_p}\Bigr)\\
&\leq\exp\Bigl(-\frac{tz^p}{n^{1-p/q}}\Bigr)\Ex\exp\Bigl(t\frac{S_q^{p/q}}{\Ex S_p}\Bigr),
\end{align*}
where we also used the independence property in \Cref{thm:cone} in the last step. Next, we observe that $\Ex S_p=n$ by \Cref{lem:MOMENTSpGenGAUSSIAN}. Moreover from \cite[Corollary 3]{N} it is known that there exists a constant $c\in(0,\infty)$ only depending on $p$ and $q$ such that
\begin{equation*}
\Ex\exp\bigl(tS_q^{p/q}\bigr) \leq n^{1-p/q}\bigl(1-ct\bigr)^{-n^{p/q}}
\end{equation*}
as long as $0<t<1/c$. Thus, choosing $t=\frac{n}{c}-\frac{n}{z^p}$ we arrive at
\begin{equation*}
\Pr(n^{1/p-1/q}\norm{X}_q>z) \leq n^{1-p/q}\Bigl(\frac{ez^p}{c}\Bigr)^{n^{p/q}}\exp(-cn^{p/q}z^p).
\end{equation*}
This implies the result.
\end{proof}

\subsection{Recent developments}\label{subsec52:RecentDevelopmentsLDPs}

\subsubsection{The LDP counterpart to Schechtman-Zinn}

After having presented the classical Schechtman-Zinn large deviation inequality, we turn now to a LDP counterpart. The next result is a summary of the results presented in from \cite[Theorems from 1.2 to 1.5]{KPT}. The speed and the rate function in its part 4 resembles the right hand side of the inequality in \Cref{thm:SZInequality}.

\begin{theorem}
	Let \(n\in\N\), \(p\in[1,\infty] \), \(q\in[1,\infty)\) and \(X\sim\nu_p^n\). Define the sequence 
	\begin{equation*}
	\norm{\XX}\coloneqq(n^{1/p-1/q}\qnorm{X})_{n\in\N}.
	\end{equation*}
	\begin{enumerate}
		\item If \(q<p<\infty\), then \(\norm{\XX}\) satisfies an LDP with speed \(n\) and good rate function 
		\begin{equation*}
			\rate_{\norm{\XX }}(x)=\begin{cases}
					\inf\{ \rate_1(x_1)+\rate_2(x_2):x=x_1x_2, x_1\ge 0, x_2\ge 0\}&\caseif x\ge 0,\\
					+\infty&\other.
				\end{cases}
		\end{equation*}	
		Here
		\begin{equation}
		\label{eq:rate1}
			\rate_1(x)=\begin{cases}
			-\log(x)&\caseif x\in(0,1],\\
			+\infty&\other,
			\end{cases}
		\end{equation}
		and
		\begin{equation*}
			\rate_2(x)=\begin{cases}
				\inf\{\Lambda^*(y,z):x=y^{1/q}z^{-1/p}, y\ge 0, z\ge 0 \}&\caseif x\ge 0\\
				+\infty&\other,
			\end{cases}
		\end{equation*}
		where \(\Lambda^* \) is the Fenchel-Legendre transform of the function 
		\begin{equation*}
			\Lambda(t_1,t_2)\coloneqq\log\int_0^{+\infty} \oneover{p^{1/p}\Gamma(1+1/p)}e^{t_1 s^q+(t_2-1/p) s^p }\de s,\qquad (t_1,t_2)\in\R\times\Bigl(-\infty,\oneover{p}\Bigr).
		\end{equation*}
	\item If \(q<p=\infty\), then \(\norm{\XX}\) satisfies an LDP with speed \(n\) and good rate function 
	\begin{equation*}
		\rate_{\norm{\XX }}(x)=\begin{cases}
		\Psi^*(x)&\caseif x\ge 0,\\
		+\infty&\other,
		\end{cases}
	\end{equation*}
	where \(\Psi^*\) is the Fenchel-Legendre transform of the function 
	\begin{equation*}
		\Psi(t)\coloneqq\int_0^1 e^{t s^q}\de s,\qquad t\in\R.
	\end{equation*}
	\item If \(p=q\), then \(\norm{\XX}\) satisfies an LDP with speed \(n\) and good rate function \(\rate_1 \) defined in \Cref{eq:rate1}.
	\item If \(p<q\), then \(\norm{\XX} \) satisfies an LDP with speed \(n^{p/q}\) and good rate function
	\begin{equation*}
		\rate_{\norm{\XX }}(x)=\begin{dcases}
			\oneover{p}\bigl(x^q-m_{p,q}^q\bigr)^{p/q}&\caseif x\ge m_{p,q},\\
			+\infty&\other.
		\end{dcases}
	\end{equation*}
	\end{enumerate}
	
\end{theorem}

\subsubsection{LDPs for projections of $\ell_p^n$-balls -- $1$-dimensional projections}

We turn now to a different type of large deviation principles. More precisely, we consider random projections of points uniformly distributed in an $\ell_p^n$-ball or distributed according to the corresponding cone probability measure onto a uniform random direction. The following result is a summary of from \cite[Theorems 2.2,2.3]{GKR2017}. The proof of the first part follows rather directly from Cramér's theorem (\Cref{thm:Cramér}) and the contraction principle (\Cref{prop:contraction principle}), the second part is based on large deviation theory for sums of stretched exponentials. 

\begin{theorem}
	Let \(n\in\N\) and \(p\in[1,\infty) \). Let \(X\sim\nu_p^n \) or \(X\sim\mu_p^n \) and \(\Theta\sim\sigma_2^n \) be independent random vectors. Consider the sequence
	\begin{equation*}
		\WW\coloneqq(n^{1/p-1/2}\scalar{X}{\Theta})_{n\in\N}.
	\end{equation*}
	\begin{enumerate}
		\item If \(p\ge 2\), then \( \WW\) satisfies an LDP with speed \(n\) and good rate function 
		\begin{equation*}
			\rate_\WW(w)=\inf\{\Phi^*(\tau_0,\tau_1,\tau_2):w=\tau_0^{-1/2}\tau_1^{\vphantom{1/2}}\tau_2^{-1/p}, \tau_0>0, \tau_1\in\R, \tau_2>0\},
		\end{equation*}
		where \(\Phi^* \) is the Fenchel-Legendre transform of 
		\begin{equation*}
			\Phi(t_0,t_1,t_2)\coloneqq\log\int_\R\int_\R e^{t_0 z^2+t_1 zy+t_2\abs{z}^p }f_2(z)f_p(y)\de z\de y,\qquad t_0,t_1,t_2\in\R.
		\end{equation*}
		\item If \(p< 2\), then \( \WW\) satisfies an LDP with speed \(n^{2p/(2+p)}\) and good rate function 
			\begin{equation*}
					\rate_\WW(w)=\frac{2+p}{2p}\abs{w}^{2p/(2+p)}.
			\end{equation*}
	\end{enumerate}
\end{theorem}
\begin{proof}
Let us sketch the proof for the case that $p>2$, by leaving out any technical details. For this, let $Z_1,\ldots,Z_n$ be $p$-generalized Gaussian random variables, $G_1,\ldots,G_n$ be Gaussian random variables and \(U\) be a uniform random variable over \([0,1]\). Also assume that all the aforementioned random variables are independent. Also put $Z\coloneqq(Z_1,\ldots,Z_n)$ and $G\coloneqq(G_1,\ldots,G_n)$. When \(X\sim\mu_p^n\), by \Cref{thm:cone}, we can state that for each $n\in\N$ the target random variable $n^{1/p-1/2}\scalar{X}{\Theta}$ has the same distribution as
\begin{equation}
\label{eq:coneX}
n^{1/p-1/2}\frac{\sum\limits_{i=1}^nG_iZ_i}{\norm{G}_2\pnorm{Z}} = \frac{\oneover{ n}\sum\limits_{i=1}^nG_iZ_i}{\Bigl( \oneover{n}\sum\limits_{i=1}^n\abs{G_i}^2\Bigr)^{1/2}\Bigl(\frac{1} {n}\sum\limits_{i=1}^n\abs{Z_i}^p\Bigr)^{1/p}}.
\end{equation}
Note that \(\Phi\) is finite whenever \(p<2\), \(t_0<1/2\), \(t_1\in\R\) and \(t_2<1/p\). Then,
Cramér's theorem (\Cref{thm:Cramér}) shows that the \(\R^3\text{-valued}\) sum
\begin{equation*}
\oneover{n}\sum_{i=1}^n\bigl(\abs{G_i}^2,G_iZ_i,\abs{Z_i}^p\bigr)
\end{equation*}
satisfies an LDP with speed $n$ and rate function $\Phi^*$. Applying the contraction principle (\Cref{prop:contraction principle}) to the function $F(x,y,z)=x^{-1/2}yz^{-1/p}$ yields the LDP for $\WW$ with speed $n$ and the desired rate function $\rate_{\mathbf{W}}$. Once the LDP is proven for the cone measure, it can be pushed to the case of the uniform measure. By \Cref{thm:cone}, multiplying the expression in \Cref{eq:coneX} by \(U^{1/n}\), we obtain a random variable distributed according to \( \nu_p^n\). It is proven in \cite[Lemma 3.2]{GKR2017} that multiplying by \(U^{1/n}\) every element of the sequence \(\mathbf{W}\), we obtain a new sequence of random variables that also satisfies an LDP with the same speed and the same rate function as \(\mathbf{W}\). On the other hand, when \(p<2\), \(\Phi(t_0,t_1,t_2)=\infty\) for any \(t_1\neq 0\), hence suggesting that in this case the LDP could only occur at a lower speed than \(n\).
\end{proof}

\subsubsection{LDPs for projections of $\ell_p^n$-balls -- the Grassmannian setting}

Finally, let us discuss projections to higher dimensional subspaces, generalizing thereby the set-up from the previous section. We adopt the Grassmannian setting and consider the $2$-norm of the projection to a uniformly distributed random subspace in the Grassmannian $\mathbb{G}_n^k$ of $k$-dimensional subspaces of $\R^n$ of a point uniformly distributed in the $\ell_p^n$-unit ball. Since we are interested in the asymptotic regime where $n\to\infty$, we also allow the subspace dimension $k$ to vary with $n$. However, in order to keep our notation transparent, we shall nevertheless write $k$ instead of $k(n)$. The next result is the collection of \cite[Theorems 1.1,1.2]{APT}.

\begin{theorem}
	Let \(n\in\N\). Fix \(p\in[1,\infty]\) and a sequence \(k=k(n)\in\{1,\ldots,n-1\}\) such that the limit \(\lambda\coloneqq\lim_{n\to\infty} (k/n)\) exists. Let \(P_E X\) be the orthogonal projection of a random vector \(X\sim\nu_p^n\) onto a random independent linear subspace \(E\sim\eta_k^n\). Consider the sequence
	\begin{equation*}
		\bPEX\coloneqq(n^{1/p-1/2}\norm{P_E X }_2)_{n\in\N}.
	\end{equation*}
	\begin{enumerate}
	\item	If \(p\ge 2\), then \(\bPEX \) satisfies an LDP with speed \(n\) and good rate function
	\begin{equation*}
	\rate_{\bPEX}(y)\coloneqq\begin{dcases}
		\inf_{x> y}\Big[{\frac{\lambda}{2}}\log\Bigl(\frac{\lambda x^2}{ y^2}\Bigr)+\frac{1-\lambda}{2}\log\Bigl(\frac{1-\lambda}{ 1-y^2x^{-2}}\Bigr)+\mathcal{J}_p(x)\Bigr] &\!\!\!\!\caseif y>0,\\
		\mathcal{J}_p(0) &\!\!\!\!\caseif y=0,\, \lambda\in(0,1],\\
		\inf\limits_{x\geq 0}\mathcal{J}_p(x) &\!\!\!\!\caseif y=0,\, \lambda=0,\\
		+\infty &\!\!\!\!\caseif y<0\,,
		\end{dcases}
	\end{equation*}
	where we use the convention \(0\log 0\coloneqq 0\) and for \(p\neq\infty\) we have
	\begin{equation*}
	\mathcal{J}_p(y)\coloneqq\inf_{\substack{x_1, x_2>0\\ x_1^{1/2}x_2^{-1/p}=y}}\mathcal{I}_p^*(x_1,x_2),\qquad y\in\R\,,
	\end{equation*}
	and $\mathcal{I}_p^*(x_1,x_2)$ is the Fenchel-Legendre transform of
	\begin{equation*}
	\rate_p(t_1,t_2)\coloneqq\log\int_{\R}e^{t_1x^2+t_2\abs{x}^p}f_p(x)\de x,\qquad (t_1,t_2)\in\R\times\Bigl(-\infty,\oneover{p}\Bigr).
	\end{equation*}
	 For $p=\infty$, we write $\mathcal{J}_\infty(y)\coloneqq\rate_\infty^*(y^2)$ with $\rate_\infty^*$ being the Fenchel-Legendre transform of $\rate_\infty(t)\coloneqq\log\int_0^1e^{tx^2}\de x$.
	 \item If \(p<2\) and \(\lambda>0\), then \(\bPEX \) satisfies and LDP with speed \(n^{p/2}\) and good rate function
	 \begin{equation*}
	 \rate_{\bPEX}(y)\coloneqq\begin{dcases}
	 	\oneover{p}\Bigl(\frac{y^2}{\lambda}-m\Bigr)^{p/2}&\caseif y\ge\sqrt{\lambda m_p}\,,\\
	 	+\infty &\other,
	 \end{dcases}
	 \end{equation*}
	 where \(m_p\coloneqq p^{p/2}\Gamma(1+3/p)/(3\Gamma(1+1/p))\).
		\end{enumerate}
\end{theorem}

Let us emphasize that the proof of this theorem is in some sense similar to its $1$-dimensional counterpart that we have discussed in the previous section. However, there are a number of technicalities that need to be overcome when projections to high-dimensional subspaces are considered. Among others, one needs a new probabilistic representation of the target random variables. In fact, the previous theorem heavily relies on the following probabilistic representation, proved in \cite[Theorem 3.1]{APT} for the case \(X\sim\nu_n^p\). We shall give a proof here for a more general set-up, which might be of independent interest.

\begin{theorem}
		Let \(n\in\N\), \(k\in\{1,\ldots,n \}\) and \(p\in[1,\infty]\). Let \(X\) be a continuous \(p\text{-radial}\) random vector in \(\R^n\) and \(E\sim\eta_k^n\) be a random \(k\text{-dimensional}\) linear subspace. Let \(Z=(Z_1,\ldots,Z_n)\) and \(G=(G_1,\ldots,G_n)\) having i.i.d. coordinates, distributed according to the densities \(f_p\) and \(f_2\), respectively. Moreover, let \(X\), \(E\), \(Z\) and \(G\) be independent. Then
		\begin{equation*}
			\norm{P_E X}_2\disteq \pnorm{X}\frac{\norm{Z}_2}{\pnorm{Z} }\frac{\norm{(G_1,\ldots,G_k)}_2}{\norm{G}_2}.
		\end{equation*}
	\begin{proof}
		Fix a vector \(x\in\R^n\). By construction of the Haar measure \(\eta_k^n\) on \(\mathbb{G}_k^n\) and uniqueness of the Haar measure $\eta$ on $\mathbb{O}(n)$, we have that, for any $t\in\R$,
		\begin{equation*}
		\begin{split}
		\eta_k^n (E\in \mathbb{G}_k^n: \norm{P_E x}_2\geq t)&=\eta(T\in \mathbb O(n): \norm{P_{T E_0} x}_2\geq t)\\
		&=\eta(T\in \mathbb O(n): \norm{P_{E_0} Tx}_2\geq t)\\
		&=\eta\bigl(T\in \mathbb O(n): \norm{x}_2\norm[\big]{P_{E_0} T(x/\norm{x}_2)}_2\geq t\bigr),
		\end{split}
		\end{equation*}
		where $E_0\coloneqq\mathrm{span}(\{e_1,\ldots,e_k\})$. Again, by the uniqueness of the Haar measure $\sigma_2^n$ on $\mathbb{S}_2^{n-1}$, \(T(x/\norm{x}_2)\sim\sigma_2^n\), provided that $T\in\mathbb{O}(n)$ has distribution $\eta$. Thus,
	\[
	\eta\Bigl(T\in \mathbb O(n): \norm{x}_2\norm[\Big]{P_{E_0} T\Bigl(\frac{x}{\norm{x}_2}\Bigr)}_2\geq t\Bigr)=\sigma_2^n(u\in \mathbb{S}_2^{n-1}: \norm{x}_2\norm{P_{E_0} u}_2\geq t)\,.
	\]
	By \Cref{thm:cone}, $G/\norm{G}_2\sim\sigma_2^n$. Thus,
	\[
	\sigma_2^n (u\in \mathbb{S}_2^{n-1}: \norm{x}_2\norm{P_{E_0} Tu}_2\geq t)=\Pr\Bigl(\norm{x}_2\,\frac{\norm{P_{ E_0}G}_2 }{\norm{G}_2}\geq t\Bigr).
	\]
Therefore, if $E\in \G^n_k$ is a random subspace independent of $X$ having distribution $\eta_k^n$, and $G$ is a standard Gaussian random vector in $\R^n$ that is independent of $X$ and $E$, we have that
	\begin{equation*}
\Pr_{(X,E)}\bigl((x,F)\in\R^n \times\G_k^n:\norm{P_F x}_2\geq t\bigr)=
\Pr_{(X,G)}\Bigl((x,g)\in\R^n\times\R^n:\norm{x}_2\,\frac{\norm{P_{ E_0}g}_2 }{\norm{g}_2}\geq t\Bigr).
\end{equation*}
	Here, $\Pr_{(X,E)}$ denotes the joint distribution of the random vector $(X,E)\in\R^n\times\G_{k}^n$, while $\Pr_{(X,G)}$ stands for that of $(X,G)\in\R^n\times\R^n$. By \Cref{prop:SZGeneralK}, \(X\) has the same distribution as \(\pnorm{X}Z/\pnorm{Z}\). Therefore,
	\begin{equation*}
	\begin{split}
	&\Pr_{(X,G)}\Bigl((x,g)\in\R^n\times\R^n:\norm{x}_2\,\frac{\norm{P_{ E_0}g}_2 }{\norm{g}_2}\geq t\Bigr)\\
	&=\Pr_{(X,Z,G)}\Bigl((x,z,g)\in\R^n\times\R^n\times\R^n:\pnorm{x}\frac{\norm{z}_2}{\pnorm{z}}\frac{\norm{P_{ E_0}g}_2 }{\norm{g}_2}\geq t\Bigr)
	\end{split}
	\end{equation*}
	with $\Pr_{(X,Z,G)}$ being the joint distribution of the random vector $(X,Z,G)\in\R^n\times\R^n\times\R^n$. Consequently, we conclude that the two random variables \(\norm{P_E X}_2\) and \(\pnorm{X}\frac{\norm{Z}_2}{\pnorm{Z} }\frac{\norm{P_{E_0} G}_2}{\norm{G}_2}\)	have the same distribution.
	\end{proof}
	\end{theorem}

\begin{remark}
Let us remark that in his PhD thesis, Kim \cite{K} was recently able to extend the results from \cite{APT} and \cite{GKR2017} to more general classes of random vectors under an asymptotic thin-shell-type condition in the spirit of \cite{ABP2003} (see \cite[Assumption 5.1.2]{K}). For instance, this condition is satisfied by random vectors chosen uniformly at random from an Orlicz ball. 
\end{remark}

\subsubsection{Outlook -- the non-commutative setting}

The body of research on large deviation principles in asymptotic geometric analysis, which we have just described above, is complemented by another paper of Kim and Ramanan \cite{KimRamanan}, in which they proved an LDP for the empirical measure of an $n^{1/p}$ multiple of a point drawn from an $\ell_p^n$-sphere with respect to the cone or surface measure. The rate function identified is essentially the so-called relative entropy perturbed by some $p$-th moment penalty (see \cite[Equation (3.4)]{KimRamanan}). 

While this result is again in the commutative setting of the $\ell_p^n$-balls, Kabluchko, Prochno, and Th\"ale \cite{KPT2018b} recently studied principles of large deviations in the non-commutative framework of self-adjoint and classical Schatten $p$-classes. The self-adjoint setting is the one of the classical matrix ensembles which has already been introduced in Subsection \ref{subsubsec: outlook non-commutative schechtman-schmuckenschlaeger} (to avoid introducing further notation, for the case of Schatten trace classes we refer the reader to \cite{KPT2018b} directly).
In the spirit of \cite{KimRamanan}, they proved a so-called Sanov-type large deviations principles for the spectral measure of $n^{1/p}$ multiples of random matrices chosen uniformly (or with respect to the cone measure on the boundary) from the unit balls of self-adjoint and non self-adjoint Schatten $p$-classes where $0< p \leq +\infty$. The good rate function identified and the speed are quite different in the non-commutative setting and the rate is essentially given by the logarithmic energy (which is the negative of Voiculescu's free entropy introduced in \cite{V1993}). Interestingly also a perturbation by a constant connected to the famous Ullman distribution appears. This constant already made an appearance in the recent works \cite{KPT2018,KPT2018a}, where the precise asymptotic volume of unit balls in classical matrix ensembles and Schatten trace classes were computed using ideas from the theory of logarithmic potentials with external fields.

The main result of \cite{KPT2018b} for the self-adjoint case is the following theorem, where we denote by $\mathcal M(\R)$ the space of Borel probability measures on $\R$ equipped with the topology of weak convergence. On this topological space we consider the Borel $\sigma$-algebra, denoted by $\mathcal B(\mathcal M(\R))$.

\begin{theorem}\label{thm:sanov type ldp schatten}
Fix $p\in(0,\infty)$ and $\beta\in\{1,2,4\}$. For every $n\in\N$, let $Z_n$ be a random matrix chosen according to the uniform distribution on $\mathbb{B}_{p,\beta}^n$ or the cone measure on its boundary. Then the sequence of random probability measures
\begin{equation*}
\mu_n= \frac{1}{n}\sum_{i=1}^n \delta_{n^{1/p}\lambda_i(Z_n)},
\qquad n\in\N,
\end{equation*}
satisfies an LDP on $\mathcal M(\R)$ with speed $n^2$ and good rate function $\mathcal I:\mathcal M(\R) \to [0,+\infty]$ defined by
\begin{equation*}
\label{eq:J_def_rate_funct}
\mathcal I(\mu) =
\begin{cases}
- \frac \beta 2 \int_{\R}\int_{\R} \log\abs{x-y} \, \mu(\de x)\,\mu(\de y) + \frac{\beta}{2p} \log\Bigl(\frac{\sqrt{\pi}p \Gamma(\frac{p}{2})}{2^p\sqrt{e}\Gamma(\frac{p+1}{2})}\Bigr) &\caseif \int_{\R}\abs{x}^p\mu(\de x) \leq 1,\\
+\infty &\caseif\int_{\R}\abs{x}^p\mu(\de x) > 1.
\end{cases}
\end{equation*}
\end{theorem}
Let us note that the case $p=+\infty$ as well as the case of Schatten trace classes is also covered in that paper (see \cite[Theorems 1.3 and 1.5]{KPT2018b}). The proof of Theorem \ref{thm:sanov type ldp schatten} requires to control \emph{simultaneously} the deviations of the empirical measures and their $p$-th moments towards arbitrary small balls in the product topology of the weak topology on the space of probability measures and the standard topology on $\R$. It is then completed by proving exponential tightness. Moreover, they also use the probabilistic representation for random points in the unit balls of classical matrix ensembles which they have recently obtained in \cite{KPT2018a}. We close this survey by saying that as a consequence of the LDP in Theorem \ref{thm:sanov type ldp schatten}, they obtained that the spectral measure of $n^{1/p} Z_n$ converges weakly almost surely to a non-random limiting measure given by the Ullman distribution, as $n\to\infty$ (see \cite[Corollary 1.4]{KPT2018b} for the self-adjoint case and \cite[Corollary 1.6]{KPT2018b} for the non-self-adjoint case). 

 \printbibliography[heading=bibintoc]
 
 \bigskip
 
 \bigskip
 
 \bigskip
 
 \bigskip
 
 Joscha Prochno: Institut für Mathematik und Wissenschaftliches Rechnen, Karl-Franzens-Universität Graz, Austria
 
 \textit{E-mail address:} \nolinkurl{joscha.prochno@uni-graz.at}
 
 \medskip
 
 Christoph Thäle: Fakultät für Mathematik, Ruhr-Universität Bochum, Germany
 
 \textit{E-mail address:} \nolinkurl{christoph.thaele@rub.de}
 
  \medskip
 
  Nicola Turchi: Fakultät für Mathematik, Ruhr-Universität Bochum, Germany
 
 \textit{E-mail address:} \nolinkurl{nicola.turchi@rub.de}
\end{document}